\documentclass{amsart}
\usepackage{amsmath, amsthm, amssymb}
\usepackage{verbatim}
\usepackage{tikz}
\usetikzlibrary{matrix}

\usepackage[mathscr]{eucal} 
\usepackage{mathrsfs} 
\usepackage{cancel}

\renewcommand{\d}{\partial}

\newcommand{\ii}{\sqrt{-1}}
\newcommand{\wt}[1]{\widetilde{#1}}

\newcommand{\wh}[1]{\widehat{#1}}

\newcommand{\veps}{\varepsilon}
\newcommand{\vphi}{\varphi}

\newcommand{\al}{\alpha}

\newcommand{\la}{\lambda}
\newcommand{\om}{\omega}
\newcommand{\te}{\theta}

\newcommand{\Om}{\Omega}
\newcommand{\De}{\Delta}

\newcommand{\na}{\nabla}

\newcommand{\cA}{\mathcal{A}}
\newcommand{\cB}{\mathcal{B}}

\newcommand{\cE}{\mathcal{E}}

\newcommand{\cS}{\mathcal{S}}

\newcommand{\supp}{\mbox{supp }}

\newcommand{\bR}{\mathbb{R}}

\newcommand{\bC}{\mathbb{C}}

\newcommand{\rB}{{\bf B}}

\newcommand{\ul}[1]{\underline{#1}}

\newtheorem{thm}{Theorem}
\newtheorem{prop}[thm]{Proposition}
\newtheorem{lem}[thm]{Lemma}
\newtheorem{cor}[thm]{Corollary}

\theoremstyle{definition}

\newtheorem{remark}[thm]{Remark}

\numberwithin{thm}{section}
\numberwithin{equation}{section}
\renewcommand{\[}{\begin{equation}}
\renewcommand{\]}{\end{equation}}

\newcommand{\wed}{\wedge}

\newcommand{\ov}[1]{\overline{#1}}

\title[Monge-Amp\`ere type equations on Hermitian manifold with boundary]{Weak solutions to Monge-Amp\`ere type equations on compact Hermitian manifold with boundary}

\author{S\l awomir Ko\l odziej and Ngoc Cuong Nguyen} 
\address{Faculty of Mathematics and Computer Science, Jagiellonian University, \L ojasiewicza 6, 30-348 Krak\'ow, Poland}
\email{slawomir.kolodziej@im.uj.edu.pl}
\address{Department of Mathematical Sciences, KAIST, 291 Daehak-ro, Yuseong-gu, Daejeon 34141, South Korea}
\email{cuongnn@kaist.ac.kr}
\begin{document}
\maketitle

%\begin{center}
%\begin{flushright}
\centerline{ {\em  Dedicated to the memory of Nessim Sibony} }
%\end{flushright}
%\end{center}

\bigskip
\bigskip

\begin{abstract} We prove the bounded subsolution theorem for the complex Monge-Amp\`ere type equation, with the right hand side being a positive Radon measure, on a compact Hermitian manifold with boundary. 
\end{abstract}

\section{Introduction}

Let $(\ov M,\om)$ be a smooth  compact $n$-dimensional Hermitian manifold with the non-empty boundary $\d M$. In \cite{KN21} we studied weak quasi-plurisubharmonic solutions of the Dirichlet problem for the complex Monge-Amp\`ere equation.  In this paper we extend those results to  the complex Monge-Amp\`ere type equation where the right hand side depends also on the solution.

Let $\mu$ be a positive Radon measure on $M = \ov M \setminus \d M$. Suppose that $F(u,z): \bR \times  M \to \bR^+$ is  a non-negative function.
 %which is continuous and non-decreasing in the first variable and $\mu$-measurable in the other one. 
Let $\vphi \in C^0(\d M)$. We consider the Dirichlet problem
\[\label{eq:DP}\begin{cases}
	u\in PSH(M,\om) \cap L^\infty(\ov M), \\
	(\om + dd^c u)^n = F(u,z) \mu,\\
	 u = \vphi \quad\text{on } \d M.
\end{cases}\]
When $M$ is a bounded strictly pseudoconvex domain in $\bC^n$ the problem for plurisubharmonic functions ($\om$ is just the zero form) was studied  by Bedford and Taylor \cite{BT77} for $\mu = dV_{2n}$ the Lebesgue measure, $F \in C^0(\bR \times \ov M)$ and $F^{1/n}$  convex and nondecreasing in $u$. Generalizations for weak solutions were done by Cegrell and Ko\l odziej (\cite{Ce84}, \cite{Ko00}, \cite{CK06}) with more general function $F(u,z)$ and a positive Radon measures $\mu$. Classical smooth solutions were obtained in  \cite{CKNS85}.  
On the product of a compact K\"ahler manifold and an annulus this problem, with $F \equiv 0$, is the geodesic equation on the space of K\"ahler potentials of the manifold. For $F(u,z) = e^{\la u}$ with $\la \in \bR$ one obtains  the K\"ahler-Einstein metric equations as in Cheng and Yau \cite{CY80}, where the authors
obtained complete  K\"ahler-Einstein metrics, but then the potentials tend to infinity when the argument approaches the boundary.
 Berman \cite{Be19} discovered that one may use the solution of a family of Monge-Amp\`ere type equation to  approximate various envelopes of quasi-plurisubharmonic functions. This approximation process is very useful in studying regularity of  global envelopes (see  \cite{CZ19}, \cite{GLZ19}, \cite{To19}).

On a compact Hermitian manifold with boundary, if  $F(u,z) = 1$ and $\mu$ is a positive Radon measure we obtained weak solutions of the problem under the hypothesis that a subsolutions exists, \cite{KN21}. To  deal with a general (non-K\"ahler) Hermitian metric $\om$ we needed to adapt a suitable comparison principle from \cite{KN1}. Furthermore, we can no longer rely on  assumption that  the boundary is pseudoconvex. We bypassed this by employing the Perron envelope and the existence of a bounded subsolution. Here we use similar strategy to obtain the following result.
\medskip

\noindent
{\bf Theorem.} (c.f. Theorem~\ref{thm:main})  {\em Assume that $F(u,z)$ is a bounded nonnegative function which is continuous and non-decreasing in the first variable and $\mu$-measurable in the other one.  Let $\mu$ be a positive Radon measure which is locally dominated by Monge-Amp\`ere measures of bounded plurisubharmonic functions. Then the Dirichlet problem \eqref{eq:DP} has a solution if and only if there is a bounded subsolution $\ul u \in PSH(M,\om) \cap L^\infty(\ov M)$ satisfying $\lim_{z\to x} u (z) = \vphi(x)$ for every $x\in \d M$ and \[\notag \label{eq:intro-subsol}(\om+dd^c \ul u)^n \geq F(\ul u, z) \mu \quad\text{on } M.\]
}

In many  cases of interest the extra assumption on $\mu$ (the local domination by Monge-Amp\`ere measures of bounded plurisubharmonic functions)   is always satisfied once the subsolution $\ul u$ exists, for example  $F(t,z) = e^{\la t}$ with $\la \in \bR$ or $\mu = \om^n$. If  $F(t,z) = e^{\la t}$ with $\la>0$, then we also have the uniqueness of the solution (Corollary~\ref{cor:KE}). 

\medskip

{\em Organization.}
In Section~\ref{sec:MA-local} we study the problem (\ref{eq:DP})  in a bounded strictly pseudoconvex domain in $\bC^n$. We give a partial generalization of a result in \cite{CK06} and study the convergence in capacity in dimension $n=2$. In Section~\ref{sec:DP} we prove the main theorem. We also give  a sketch of the proof of
 H\"older continuity of solutions of the Laplace equation on compact Hermitian manifold with boundary in the appendix.

\medskip

{\em Acknowledgement.} The first author is partially supported by  grant  no. \linebreak 2021/41/B/ST1/01632 from the National Science Center, Poland. The second author is  partially supported by the start-up grant G04190056 of KAIST and the National Research Foundation of Korea (NRF) grant  no. 2021R1F1A1048185.

\section{Monge-Amp\`ere type equations and stability of solutions}
\label{sec:MA-local}

In this section we first study the problem~\eqref{eq:DP} in the special case $M \equiv \Om$ a bounded strictly pseudoconvex domain in $\bC^n$. This will provide the construction of the lift of quasi-plurisubharmonic function in the Perron envelope method for the manifold case.

Let us recall the following version of the comparison principle for a background Hermitian metric, see \cite[Theorem~3.1]{KN1}.

\begin{lem} \label{lem:local-CP} Let $\Om$ be a bounded open set in $\bC^n$. Fix $0<\te<1$. Let $u,v \in PSH(\Om, \om) \cap L^\infty(\Om)$ be such that $\liminf_{z\to \d \Om} (u-v) \geq 0$. Suppose that $-s_0 = \sup_{\Om}(v-u) >0$ and $\om + dd^c v \geq \te \om$ in $\Om$. Then, for any $0< s < \te_0:= \min \{\frac{\te^n}{16 {\bf B}}, |s_0| \}$, 
$$
	\int_{\{u< v + s_0 +s\}} (\om + dd^c v)^n  \leq \left( 1+ \frac{s{\bf B}}{\te^n} C_n \right) \int_{\{u<v+s_0 +s\}} (\om+ dd^c v)^n,
$$ 
where $C_n$ is a dimensional constant and  ${\bf B}>0$ is a constant such that on $\ov \Om$, 
$$	-\rB \om^2 \leq 2n dd^c \om \leq \rB \om^2, \quad -\rB \om^3 \leq 4n^2 d\om \wed d^c \om \leq \rB \om^3.
$$
\end{lem}

It gives a useful comparison principle for solutions of Monge-Amp\`ere type equations. We shall use it frequently in the paper.

\begin{prop} \label{prop:uniqueness} Let $\Om$ be a bounded open set in $\bC^n$. Let $u, v \in PSH(\Om, \om) \cap L^\infty(\Om)$ be such that $\liminf_{z\to \d\Om} (u-v)(z) \geq 0$.  Let $F(t,z): \bR \times \Om \to \bR^+$ be a non-negative function which is non-decreasing in $t$ and $d\mu$-measurable in $z$. Suppose $\mu \leq \nu$ as measures and
$$
	(\om + dd^c u)^n = F(u,z) \mu, \quad (\om + dd^c v)^n = F(v, z) \nu
$$
in $\Om$. Then, $u \geq v$ on $\Om$.
\end{prop}

\begin{proof} We may additionally assume that $\liminf_{z\to \d\Om} (u-v) \geq 2a>0$. The general case follows after replacing $u$ by $u+2a$ and letting $a \to 0$. Thus $\{u< v\} \subset \Om' \subset \subset \Om$ for some open set $\Om'$. By subtracting from $u, v$ a constant $C$ and replacing $F(t,z)$  by $ F(t+C, z)$ we may also assume that $u, v \leq 0$.  Arguing by contradiction, suppose that $\{u<v\}$ was non-empty.  Since $\Om$ is bounded, there exists a bounded, strictly plurisubharmonic function $\rho \in C^2(\ov \Om)$ such that $-C \leq \rho \leq 0$. Then one can first multiply $\rho$ by small positive constant and then fix small positive constants $\te, \tau>0$ such that 
 the set $\{u< (1+\tau)^\frac{1}{n} v +  \rho \} \subset \subset \Om$ is nonempty and
$$
	dd^c \rho \geq 2\te \om, \quad 1+ \te \geq (1+\tau)^\frac{1}{n} .
$$
Put $\wh v = (1+\tau)^\frac{1}{n}v+ \rho$. Then, 
\[\label{eq:lower-bound} \begin{aligned}
\om_{\wh v}^n 
&= \left(\om + dd^c \rho + dd^c (1+\tau)^\frac{1}{n} v \right)^n \\
&\geq \left[(1+2\theta )\om +  dd^c (1+\tau)^\frac{1}{n} v \right]^n\\
& \geq (1+\tau) \om_v^n.
\end{aligned}\]
Denote by $U(s) $ the set $\{u< \wh v + s_0 +s\}$ where $-s_0 = \sup_{\Om} (\wh v -u) >0$. Then, for $0<s<|s_0|$, 
$$
	U(s) \subset \subset \Om, \quad  \sup_{U(s)} \{\wh v +s_0 +s - u\} =s.
$$
It follows from Lemma~\ref{lem:local-CP} that  for every $0< s < \min\{\frac{\te^n}{16 {\bf B}}, |s_0|\}$,
$$
	0<\int_{U(s)}  \om_{\wh v}^n \leq \left( 1+ \frac{s {\bf B}}{\te^n} C_n\right) \int_{U(s)} \om_u^n,
$$
where the first inequality holds because $\om + dd^c \wh v \geq \te\om$.
Note that $u < \wh v$ on $U(s)$. Hence, the monotonicity of $F$ implies
$$\om_u^n = F(u, z) \mu \leq F(u, z) \nu \leq F(\wh v,z) \nu \leq F(v, z) \nu = \om_v^n.$$
Combining this and \eqref{eq:lower-bound} we get
$$
	(1+\tau) \int_{U(s)} \om_v^n \leq \left( 1+ \frac{s {\bf B}}{\te^n} C_n\right) \int_{U(s)} \om_v^n.
$$
Therefore, $0< \tau \leq s{\bf B} C_n/\te^n$, which is impossible for $s>0$ small enough. Hence, $u \geq v$ on $\Om$ and the proof is completed.
\end{proof}

Let $\Om$ be a bounded strictly pseudoconvex domain in $\bC^n$. Let $\vphi \in C^0(\d \Om)$ and $\mu$ be a positive Radon measure in $\Om$.  Suppose that $F(u,z): \bR \times \Om \to \bR^+$ is non-negative function which is continuous and non-decreasing in the first variable and $\mu$-measurable in the other one. %We extend the results in \cite{KN21} to more general Monge-Amp\`ere type equations. 
Consider the following Dirichlet problem
\[\label{eq:local-D}
\begin{cases}
	u \in PSH(\Om, \om) \cap L^\infty(\ov \Om), \\
	(\om + dd^c u)^n = F(u,z) \mu, \\
	\lim_{z\to x} u(z) = \vphi(x)\quad\text{for } x\in \Om.
\end{cases}
\]
By Proposition~\ref{prop:uniqueness} there is at most one solution to the problem. Furthermore, in the special case $\mu\equiv 0$, by \cite[Theorem~4.2]{KN1} there is a unique continuous (maximal) $\om$-plurisubharmonic ($\om$-psh for short) function $h$ solving
\[\label{eq:maximal}
	(\om + dd^c h)^n \equiv 0, \quad h = \vphi \quad\text{on } \d \Om.
\]

We now study the existence of the solution. The Cegrell class $\cE_0(\Om)$ is the set of all functions $v \in PSH(\Om) \cap L^\infty(\Om)$ satisfying
$$
	\lim_{z\to \d\Om} v(z)= 0 \quad\text{and}\quad
	\int_\Om (dd^cv)^n <+\infty.
$$
The following generalizes \cite[Theorem~1.1]{Ko00}, where the case $\om =0$ was treated.

\begin{thm}\label{thm:existence-local} Suppose that $F(t,z)$ is a bounded non-negative function which is continuous and non-decreasing in the first variable and $\mu$-measurable in the second one. Suppose that $d\mu \leq (dd^c v)^n$ for a function $v\in \cE_0(\Om)$. Then, there exists a unique bounded $\om$-psh function in $\Om$ solving 
$$\begin{aligned}
&	(\om + dd^c u)^n = F(u, z) d\mu, \\
&	\lim_{z\to x} u(z) = \vphi(x) \quad\text{for } x \in \d\Om.
\end{aligned}$$
\end{thm}

\begin{proof}
The proof follows the lines of  \cite[Theorem~1.1]{Ko00} (see also  \cite{Ce84}), which applied Schauder's fixed theorem. In the presence of the Hermitian background form
we need  results from \cite{KN21} to verify the hypothesis of that theorem.

We start with the following simple observation.
 
\begin{lem} \label{lem:compact-support}
We may assume that $\mu$  has compact support in $\Om$.
\end{lem}

\begin{proof}Suppose that the problem is solvable for compactly supported measures. Let $\{\Om_j\}_{j\geq 1}$ be an increasing exhausting sequence  of open sets that are relatively compact in $\Om$.  Then we can find a sequence of functions: $u_j \in PSH(\Om, \om) \cap L^\infty(\Om)$ with $\lim_{z\to x} u_j(z) = \vphi(x)$ solving
$$
	(\om + dd^c u_j)^n = F(u_j, z) {\bf 1}_{\Om_j} d\mu.
$$
By Proposition~\ref{prop:uniqueness} this is a decreasing sequence and
$
	v + h \leq u_j \leq h \quad \text{on } \ov\Om,
$
where $h$ is the maximal function defined in \eqref{eq:maximal}. 
Put $u = \lim u_j$. By monotone convergence theorems \cite{DK12} (see also \cite{BT82}) we have that $\om_{u_j}^n$ converges weakly to $\om_u^n$. Also $F(u_j, z) {\bf 1}_{\Om_j} d\mu$ converges weakly to $F(u, z) d\mu$ when $j\to \infty$. Thus $u$ satisfies the equation $\om_u^n = F(u, z) d\mu$.
\end{proof}

%\begin{proof}[Proof of Theorem~\ref{thm:existence-local}]

 By Lemma~\ref{lem:compact-support} we assume that $\supp \mu$ is compact in $\Om$.
Since $F$ is bounded, without loss of generality we may assume that $0 \leq F \leq 1$, as we can rescale $ \mu$ and $v$ by a positive constant.  Let $h$ be a maximal $\om$-psh in $\Om$ as in \eqref{eq:maximal}. Then, the set 
$$
	\cA = \{u \in PSH(\Om, \om): v+ h \leq u \leq h\}
$$
is a convex and  bounded set in $L^1(\Om)$ with respect to $L^1-$topology. Thus, it is a compact set. We define the map $T : \cA \to \cA$, where $T(u) =w$ is the solution of
$$
	w \in \cA, \quad (\om + dd^c w)^n = F(u, z) d\mu,
$$
where this solution is uniquely determined  by \cite[Theorem~3.1]{KN21}. 

Next, we shall verify that $T$ is continuous. Let $\{u_j\}$ be a sequence in $\cA$ such that $u_j  \to u$ in $L^1(\Om)$. Set $w = T(u)$ and $w_j = T(u_j)$. The continuity of $T$ will follow if we have 
\[ \label{eq:lim-sup-inf}
	\limsup w_j \leq w \leq \liminf w_j.
\]
Since $\cA$ is compact in $L^1(\Om)$ we may assume that $w_j$ converges to $w$ in $ L^1(\Om).$ In the arguments we often pass to a subsequence which does not affect the  proof. We also skip the renumbering of indices of those subsequences.

We need the following result which  is essentially due to Cegrell \cite{Ce98}.
\begin{lem}\label{lem:RHS-convergence}  There is a subsequence of $\{F(u_j, z)\}_{j\geq 1}$ that converges  to $F(u, z)$ in $L^1(d\mu)$.
\end{lem}  

\begin{proof}[Proof of Lemma~\ref{lem:RHS-convergence}] Since $u_j$ are uniformly bounded we may subtract a constant and assume that $u_j, u \leq 0$ for all $j\geq 1$. Then it follows from \cite[Lemma~2.1]{KN21} that $\lim \int_\Om u_j d\mu = \int_\Om u d\mu$.  Since $\mu$ is dominated by capacity, it follows from \cite[Corollary~2.2]{KN21}  that $u_j \to u$ in $L^1(d\mu)$. Passing to a subsequence we have that $u_j$ converges to $u$ almost everywhere in $d\mu$. Since $F(t, z)$ is continuous in  $t$, the sequence $F(u_j, z)$ converges $\mu$-a.e to $F(u, z)$. By the Lebesgue dominated convergence theorem we get the conclusion.
\end{proof}

After passing to  a subsequence we may also assume  that $u_j$ converges to $u$ almost everywhere in $d\mu$.  Let us check the first inequality in \eqref{eq:lim-sup-inf}. Define $\wh u_k = \inf_{j\leq k} u_j$ and  denote by $\wh w_k \in \cA$ the sequence  of solutions to
$$
	(\om + dd^c \wh w_k)^n = F(\wh u_k, z) d\mu.
$$
Since $(\om + dd^c \wh w_k)^n$ is increasing, by the comparison principle \cite[Corollary~3.4]{KN1},  $\wh w_k$ decreases to $\wh w \in \cA$.  As $u_j$ converges to $u$, we have  that $\wh u_k$ increases to $u$ (almost everywhere in $d\mu$). Therefore, 
$$
	(\om + dd^c \wh w)^n = \lim_{k\to +\infty}(\om + dd^c \wh w_k)^n = \lim_{k\to +\infty} F (\wh u_k, z) d\mu = F(u, z) d\mu,
$$
where the first equality holds by the convergence theorem in \cite{DK12} (see also  \cite{BT82}) and the last equality follows from Lemma~\ref{lem:RHS-convergence} and the Lebesgue monotone convergence theorem.
By the uniqueness of the solution $w = \wh w$. Note that $\wh w_k \geq w_k$ because $(\om+ dd^c w_k)^n \geq (\om + dd^c \wh w_k)^n$. Therefore,
$
	w = \lim \wh w_k \geq  \limsup w_k.
$

Next we prove the second inequality in \eqref{eq:lim-sup-inf}. By Hartogs' lemma for $\om$-psh functions, $u = (\limsup_{j\to +\infty} u_j)^*$. Define $\wt u_k = (\sup_{j\geq k} u_j)^*$ and $\wt w_k = T (\wt u_k)$. Since $(\om+ dd^c \wt w_k)^n = F(\wt u_k,z) d\mu$ is decreasing, again by the comparison principle \cite{KN1} one gets that the sequence $\wt w_k$ is increasing to some $\wt w \in \cA$.  Note also that
$$
	(\om +dd^c w_k)^n = F(u_k, z) d\mu \leq F (\wt u_k, z) d\mu = (\om + dd^c \wt w_k)^n.
$$
Hence, $\wt w_k \leq w_k$.
Furthermore, $\wt u_k \downarrow u$ (almost everywhere in $d\mu$). By the convergence theorem in [DK12]   and Lemma~\ref{lem:RHS-convergence} we infer $$(\om + dd^c \wt w)^n = \lim_{k\to +\infty} F(\wt u_k, z) d\mu = F(u, z) d\mu.$$
As above by the uniqueness $w = \wt w$. Then, $w = \lim \wt w_k \leq \liminf w_k.$

Thus, we conclude the continuity of $T$. The Schauder theorem says that $T$ has a fixed point $T(u) = u$. This gives the existence of a solution to the Dirichlet problem. 
\end{proof}

\begin{remark} Thanks to the first  lemma of the above proof the theorem can be  extended to the case of unbounded $F$, e.g.  $F(u, z) = |u|^{-\al}$ with $\al>0$, in the Dirichlet problem: 
$$\begin{cases} 
u \in PSH(\Om) \cap L^\infty(\ov\Om),\\
( dd^c u)^n = |u|^{-\al} d\mu, \\
u=  0 \quad\text{on } \d \Om.
\end{cases}
$$
Indeed, suppose that there exists a subslution $\ul u \in PSH(\Om) \cap L^\infty(\ov\Om)$ such that $\ul u =0$ on $\d \Om$ and 
$(dd^c \ul u)^n \geq |\ul u|^{-\al} d\mu.$ The proof above shows that for the sequence $\Om_j \uparrow \Om$ in Lemma~\ref{lem:compact-support} we can find $u_j \in PSH(\Om) \cap L^\infty(\ov\Om)$ such that 
$$	(dd^c u_j)^n = |u_j|^{-\al} {\bf 1}_{\Om_j} d\mu, \quad u_j =0  \text{ on } \d\Om.
$$
Since $|u_j|^\al (dd^c u_j)^n \leq |\ul u|^\al (dd^c \ul u)^n$,  by the comparison principle $\ul u \leq u_j \leq 0$. Similarly, $u_j$ is a uniformly bounded and decreasing sequence. Therefore, $u = \lim_{j\to\infty} u_j$ is the unique solution to  the Dirichlet problem.

Furthermore  one can obtain better regularity of the solution under stronger assumptions on $\mu$. Let us consider $\mu =f(z) dV_{2n} $ with $f \in L^1(\Om)$ and let $\rho$ be the strictly pluriharmonic defining function for $\Om$. Assume that $|\rho|^{-\al} f \in L^p(\Om)$ for some $p>1$ (e.g., when $0<\al<1$ and $f\in L^p(\Om)$  is bounded near the boundary $\d\Om$.)   Then, we have
$$
	 |\rho|^{-\al} f dV_{2n} = (dd^c w)^n, \quad w = 0 \quad\text{on } \d \Om
$$
for some H\"older continuous plurisubharmonic function $w$ (see \cite{GKZ08}). Therefore, 
$$
	f dV_{2n} = |\rho|^\al (dd^c w)^n \leq |w + \rho|^\al [dd^c (w+\rho)]^n.
$$
Therefore, $\rho+ w$ is a H\"older continuous subsolution to the Dirichlet problem. So there is a unique bounded solution $u$. Furthermore, $$|u|^{-\al} f dV_{2n} = (|\rho|/|u|)^\al |\rho|^{-\al} f dV_{2n}.$$
A simple use of H\"older inequality shows  that the measure on the right hand side is well dominated by capacity on every compact set of $\Om$. Thanks to \cite[Remark~4.5]{KN21} we get that $u$ is continuous.  Note that the H\"older continuity of $u$ is  proved in \cite{HQ19}.

The above problem appeared in \cite{BT77} and was more recently studied in \cite{Cz10} and \cite{HQ19,HQ21}. 
\end{remark}

It is well known that the convergence  of $\om$-psh functions in $L^p(\Om)$   does not imply the weak convergence of their associated Monge-Amp\`ere operators. 
However, under the assumption of a uniform bound of the Monge-Am\`ere measures this is the case, see \cite[Theorem~2.2]{CK06}. 
We give a (partial) corresponding  stability result  for Hermitian background metrics.

\begin{thm}\label{thm:stability} Suppose $d\mu  = (dd^c v)^n$ for some $v \in \cE_0(\Om)$. Let $0\leq f_j \leq 1$ be a sequence of $d\mu-$measurable functions such that $f_j d\mu $ converges weakly to $f d\mu$ as measures. Suppose that $u_j$ solves 
$$\begin{cases}
	u_j \in PSH(\Om, \om) \cap L^\infty(\ov \Om), \\
	(\om + dd^c u_j)^n = f_j d\mu, \\
	\lim_{z\to x} u_j(z) = \vphi(x)\quad\text{for } x\in \d \Om ,
\end{cases} $$
for a continuous function $\vphi$.

Then, for $u =\lim u_j$ in $L^1(\Om)$ we have $(\om + dd^c u)^n = f d\mu$. 
\end{thm}

\begin{proof} Let $h$ be the maximal function defined in \eqref{eq:maximal}. Then by the comparison principle \cite[Corollary~3.4]{KN1} we have $h+ v \leq u_j \leq h$. Thus the sequence $\{u_j\}$ is uniformly bounded and this allows to apply \cite[Lemma~3.5 c)]{KN21} and conclude  that there is a subsequence $\{u_{j_s}\}$ such that 
\[\label{eq:convergence-n}
	\lim_{j_s\to \infty} \int_\Om |u_{j_s} -u| (\om + dd^c u_{j_s})^n =0,
\] 
and moreover $\om_{u_{j_s}}^n$ converges weakly to $\om_u^n= f d\mu$, see   \cite[Lemma~3.6]{KN21}. 
This gives the  result. 
\end{proof}

In this statement we would like to have also that a subsequence of  $u_j$ converges to $u$ in capacity. Near the end of this section (Corollary~\ref{cor:capacity-convergence}) we are able to do this for $n=2.$ 
Recall that the Bedford-Taylor capacity for a Borel set $E\subset \Om$ is defined by
\[ \label{eq:BT-capacity} 
	cap (E) = cap (E,\Om) = \sup\left\{  \int_E (dd^c v)^n: v\in PSH(\Om), -1 \leq v\leq 0 \right\}.
\]
Note that there is another (natural) capacity associated with the metric $\om$ which is equivalent to the one above (see \cite[Lemma~5.6]{KN21}).

\begin{prop}\label{prop:characterization} Let $\{u_j\}$ be the sequence in Theorem~\ref{thm:stability} in the case of smooth  $\vphi$. Then, $u_j$ converges to $u$ in capacity if and only if 
\[ \label{eq:hessian-energy}
 \lim_{j\to +\infty }\int_{\Om} |u_j -u| \om_{u_j}^k \wed \om^{n-k} =0, \quad k=0,...,n.
\]
\end{prop}

\begin{proof} Suppose that $u_j \to u$ in capacity, i.e., for a fixed $\veps >0$, we have $$\lim_{j\to +\infty} cap (|u_j -u|>\veps) =0.$$
Let $g$ be a strictly plurisubharmonic function in a neighborhood of $\ov\Om$ such that $dd^c g \geq \om$ and $g= -\vphi$ on $\d\Om$. Denoting $\wh u_j = u_j+ g$, we have $\wh u_j \to \wh u= u+g$ in capacity. Let $\rho$ be a strictly plurisubharmonic defining function of $\Om$ such that $dd^c \rho \geq \om$ on $\ov \Om$. By \cite[Lemma~3.3, Corollary~3.4]{KN21}, 
\[\label{eq:uniform-mass}
	\sup_{j\geq 1} \int_{\Om} (dd^c \wh u_j)^k \wed (dd^c \rho)^{n-k} \leq C
\]
for some $C$. We can repeat the argument of \cite[Lemma~2.3]{KN21}, with the sequence $\{w_j\}_{j\geq 1}$ in the place of $\{\wh u_j\}_{j\geq 1}$,  to  obtain
$$
	\lim_{j\to +\infty} \int |u_j -u| (dd^c \wh u_j)^k \wed (dd^c \rho)^{n-k} = 0,
$$ 
This gives the proof the necessary condition.

It remains to  prove the other implication. Suppose that we have the limit \eqref{eq:hessian-energy} for all $k=0,...,n$. Note that 
$$
	cap(|u_j-u|>\veps) = cap (|\wh u_j - \wh u|>\veps) \leq cap (\wh u_j - \wh u>\veps) + cap (\wh u - \wh u_j >\veps).
$$
By Hartogs' Lemma it is easy to see that $\max\{\wh u_j, \wh u\} \to \wh u$ in capacity. Hence, 
$
	\lim_{j\to +\infty} cap (\wh u_j -\wh u > \veps) =0.
$
The remaining term is estimated as in  \cite[page 718]{CK06} via the comparison principle for plurisubharmonic functions:
$$\begin{aligned}
	cap (\wh u_j < \wh u -\veps) 
\leq \left(2/\veps\right)^n \int_{\{u_j < u -\veps/2\}}  (dd^c \wh u_j)^n \\
\leq (2/\veps)^{n+1} \int_{\Om} |u_j-u| (dd^c \wh u_j)^n.
\end{aligned}$$
Denote $\tau = dd^c g -\om \leq C \om$. Then, $dd^c \wh u_j = \om_{u_j} + \tau$. It follows that 
$$
	\int_{\Om}|u_j -u| (dd^c \wh u_j)^n = \sum_{k=0}^n \binom{n}{k}\int_\Om |u_j -u|  \om_{u_j}^k \wed \tau^{n-k}
$$
which  goes to zero by the assumption. Thus, the convergence in capacity is proved by the previous inequality.
\end{proof}

\begin{lem}\label{lem:2-dim} Suppose $n=2$. Let $u,v \in PSH(\Om, \om) \cap L^\infty(\ov \Om)$ be such that  $u\leq v$ in $\Om$ and $u = v$ near $\d \Om$. Let $-1 \leq \rho \leq 0$ be a plurisubharmonic function in $\Om$. Then,
$$
	\int_\Om  (v-u)^3 (dd^c \rho)^2 \leq 6 \int_\Om (v-u) \om_u^2 + C \int_\Om (v-u)^2 \om^2 +C E,
$$
where $C$ is a uniform constant depending only on $\Om$ and $\om$ and 
$$E =   \left(\int_{\Om} (v-u) \om_u \wed \om  \right)^\frac{1}{2} \left(\int_{\Om}  (v-u)^2\,\om^2 \right)^\frac{1}{2}. $$
\end{lem}

\begin{proof} By quasi-continuity of plurisubharmonic and $\om$-psh functions and the convergence theorems in \cite{BT82} and \cite{DK12} we may assume that all functions are smooth. Let us denote $h = v-u \geq 0$. Then, $dd^c h = \om_v - \om_u$ and $dh\wed d^ch$ is a positive (1,1)-current. By integration by parts
$$\begin{aligned}
	\int_\Om h^3 (dd^c \rho)^2 
&=	 \int_{\Om} \rho dd^c h^3 \wed dd^c \rho\\
&=	\int_\Om \rho (3 h^2 dd^c h + 6 h dh \wed d^c h) \wed dd^c \rho \\
&\leq 	3\|\rho\|_\infty  \int_{\Om} h^2 \om_u \wed dd^c \rho.
\end{aligned}$$
Using the integration by parts again
$$
	\int_\Om h^2 \om_u \wed dd^c \rho = \int_\Om \rho dd^c ( h^2 \om_u).
$$
Here,
$$ \begin{aligned}
	\rho dd^c (h^2 \om_u) 
&= 	\rho \left[dd^c h^2 \wed \om_u+ 2 dh^2 \wed d^c \om + h^2 dd^c \om \right] \\
&= 	\rho \left[2 h dd^c h \wed \om_u + 2 dh \wed d^c h \wed \om_u + 2 dh^2 \wed d^c \om + h^2 dd^c \om \right] \\
&\leq	 2 |\rho| h\, \om_u^2 + 4\rho h \,dh \wed d^c \om + \rho h^2 dd^c \om.
\end{aligned}$$
because $\rho$ is negative.  Using $-C \om^2\leq dd^c \om \leq C\om^2$ and $-1\leq \rho \leq 0$, we have
$$
	\int_\Om  h^2 \om_u \wed dd^c \rho \leq 2 \int_\Om h \om_u^2 + 4 \left| \int_\Om \rho h\, dh \wed d^c \om \right| + C \int_\Om h^2 \om^2.
$$
Thus, to complete the proof we need to estimate the middle integral on the right hand side. In fact,  using \cite[Proposition~1.4]{Ng16} we have
$$\begin{aligned}
	\left|\int_{\Om} \rho h\, d h \wed d^c \om  \right| 
&\leq  C \left(\int_{\Om}  d h \wed d^c h \wed \om \right)^\frac{1}{2} \left(\int_{\Om} |\rho|^2 h^2 \;\om^2 \right)^\frac{1}{2} \\ 
&\leq C  \left(\int_{\Om} d h \wed d^c h \wed \om \right)^\frac{1}{2} \left(\int_{\Om}  h^2\,\om^2 \right)^\frac{1}{2}. 
\end{aligned}$$ 
Note that $2 dh \wed d^c h = dd^c h^2 - 2h dd^c h \leq dd^c h^2 + 2h \om_u$.  Therefore, 
$$2\int_\Om d h \wed d^c h \wed \om \leq \int_\Om (dd^c h^2  +2h \om_u)  \wed \om .$$
By integration by parts $$\int_\Om dd^c h^2 \wed \om = \int_\Om h^2 dd^c \om^2 \leq C \int_\Om h^2 \om^2.$$
Combining these inequalities we get
$$\begin{aligned}	
\left|\int_{\Om} \rho h\, d h \wed d^c \om  \right|  &\leq C  \left(\int_{\Om} h^2\om^2 + h \om_u \wed \om  \right)^\frac{1}{2} \left(\int_{\Om}  h^2\,\om^2 \right)^\frac{1}{2} \\
&\leq C\int_{\Om}  h^2\,\om^2  + C\left(\int_{\Om} h\om_u \wed \om   \right)^\frac{1}{2} \left(\int_{\Om}  h^2\,\om^2 \right)^\frac{1}{2},
 \end{aligned}$$
 where we used an elementary inequality $(x+y)^\frac{1}{2} \leq x^\frac{1}{2} + y^\frac{1}{2}$ with $x, y\geq 0$ in the second inequality.
This  completes the proof.
\end{proof}

\begin{cor}\label{cor:capacity-convergence} If $n=2$, then there is a subsequence $\{u_{j_s}\}$ of $\{
u_j\}$ in Theorem~\ref{thm:stability} that converges to $u$ in capacity.
\end{cor}

\begin{proof}
 Fix $a>0$. Let us denote $w_s = \max\{u_{j_s}, u-1/s\}$. By Hartogs' Lemma $w_s\to u$ in capacity. Moreover, $\{|u-u_j|>2a\} \subset \{|u-w_s|>a\} \cup \{ |w_s - u_{j_s}|>a \}$. Therefore, it is enough to show that 
$
	cap(|w_s - u_{j_s}|>a) \to 0 
$
as $s \to +\infty.$
By definition we have $w_s = \max\{u_{j_s}, u- 1/s\} \geq u_{j_s}$. Then, we need to show 
$$
	cap(\{u_{j_s} < w_s - a\}) \to  0 \quad\text{as} \quad s\to +\infty.
$$
To this end, by Lemma~\ref{lem:2-dim},
$$
	\left(\frac{a}{2} \right)^3 cap (\{ u_{j_s} < w_s -a \}) \leq  \int_\Om (  w_s - u_{j_s} )  \om_{u_{j_s}}^2 + C \int_\Om (w_s - u_{j_s})^2 \om^2 + CE_s,
$$
where $$E_s =   \left(\int_{\Om} (w_s-u_{j_s}) \om_{u_{j_s}} \wed \om  \right)^\frac{1}{2} \left(\int_{\Om}  (w_s-u_{j_s})^2\,\om^2 \right)^\frac{1}{2}.$$
By the assumption $u_j \to u$ in $L^1(\Om)$  (they are uniformly bounded) and \eqref{eq:convergence-n}, the first and second integrals on the right hand side go to zero as $s$ goes to infinity. 
Note that the first factor of $E_s$ is uniformly bounded by \eqref{eq:uniform-mass}. Hence, the last term $E_s$ also goes to zero as $s\to +\infty$.
Therefore, the conclusion follows.
\end{proof}

Proposition~\ref{prop:characterization} and Corollary~\ref{cor:capacity-convergence} have their analogues on a compact Hermitian manifold without boundary. They allow, for example,  to provide  another proof of the following result of Guedj and Lu \cite[Proposition~3.4]{GL21a}. 

\begin{lem} Let $(X, \om)$ be a compact $n$-dimensional Hermitian manifold (without boundary). Then for any $A>0$,
$$
	  \inf\left\{ \int_X (\om + dd^c v)^n: v \in PSH(X,\om), -A \leq v \leq 0\right\} >0.
$$
\end{lem}

\begin{proof} By replacing $\om$ with $ \om/A$ and $v$ with  $v/A$ we may assume that $A =1$. We argue by contradiction. Suppose that there was a sequence $\{u_j\}_{j\geq 1}  \subset PSH(X, \om)$ such that $-1\leq u_j \leq 0$ and 
$$
	\lim_{j\to \infty} \int_{X} \om_{u_j}^n =0, \quad \sup_X u_j =0.
$$
By passing to a subsequence we assume that $u_j \to u$ in $L^1(X)$  and $u_j \to u$ a.e, where $$u = (\limsup_{j\to \infty} u_j)^* = \lim_{j\to \infty} (\sup_{\ell \geq j} u_\ell)^*.$$
Then, $-1\leq u \leq 0$ and $u \in PSH(X, \om)$. We will show that there exists a subsequence $\{u_{j_s}\}$ of $\{u_j\}$ such that $\om_{u_{j_s}}^n$ converges weakly to $\om_u^n$. Indeed, set 
$$
	w_{j} = \{u_j, u-1/j\}.
$$
By the Hartogs lemma $w_j$ converges to $u$ in capacity. Therefore, by the convergence  theorem in [BT82] and [DK12] 
$
	\lim_{j\to \infty} \om_{w_j}^n = \om_u^n.
$
Next, we will show that 
$$
	\int_X |u_j -u| (\om + dd^c u_j)^n \to 0	 \quad\text{and}\quad
	\int_{X} |u_j -u| (\om + dd^c w_j)^n \to 0
$$
as $j \to +\infty$. 
The first statment holds because $\int_X |u_j -u| \om_{u_j}^n \leq 2 \int_X \om_{u_j}^n \to 0$ by the assumption. The second convergence also holds because $w_j \to u$ in capacity and so the proofs of Lemma~2.1, Corollary~2.2, Lemma~2.3 in \cite{KN21} can be applied because all considered functions are uniformly bounded and $X$ is compact. 

Thus, again by \cite[Lemma~3.6]{KN21} there exists a subsequence $u_{j_s}$ such that $\om_{u_{j_s}}^n$ converges weakly to $\om_u^n$. In particular, $\int_X \om_u^n = \lim_{j_s\to +\infty} \int_X \om_{u_{j_s}}^n = 0$. Hence, $\om_u^n \equiv 0$ on $X$.  This is a contradiction with the fact that the Monge-Amp\`ere mass of a bounded $\om$-psh function is always positive \cite[Remark~5.7]{KN1}.
\end{proof}

\section{The Dirichlet problem}
\label{sec:DP}

In this section we solve the Dirichlet problem for Monge-Amp\`ere type equation under the existence of a bounded subsolution.
Let $\mu$ be a positive Radon measure on $M = \ov{M}\setminus \d M$. Let $\vphi \in C^0(\d M)$.  Let $F(t,z): \bR \times M \to \bR^+$ be a non-negative function which is non-decreasing in $t$ and $d\mu$-measurable in $z$. We consider the Dirichlet problem
\[\label{eq:Dirichlet} 
\begin{cases}
	u \in PSH(M, \om) \cap L^\infty(\ov M), \\
	(\om + dd^c u)^n = F(u, z)\mu, \\ 
	\lim_{z \to x} u(z) = \vphi(x) \quad\text{for } x\in \d M.
\end{cases}
\]
%In our previous paper we considered the case $F(t, z) = 1$ which is the complex Monge-Amp\`ere equation. Another important classes are $F(t,z) = e^{\la t }$ for $\la >0$. This is because it is related to the K\"ahler-Einstein metric equations and it plays an important role in the approximation process due to Berman \cite{?}.
A necessary condition to solve the Dirichlet problem is the existence of a subsolution, i.e., a function $\ul u \in PSH(M, \om)$  satisfying $\lim_{z\to x} u(z) = \vphi(x)$  for  $x\in \d M$ and
\[\label{eq:subsolution-bounded} 
	(\om + dd^c \ul u)^n \geq F(\ul u, z) \mu.
\]
We  say that $\mu$ is locally dominated by Monge-Amp\`ere measures of bounded plurisubharmonic functions if for each $p\in M$, there exists a  coordinate ball $B \subset \subset M$ centered at $p$ and $v \in PSH(B) \cap L^\infty(B)$ such that 
\[\label{eq:local-dominated}
	\mu_{|_B} \leq (dd^c v)^n.
\]
By the subsolution theorem from \cite{Ko95} on a smaller ball $B' \subset \subset B$  we can choose $v$ in the Cegrell class  $\cE_0(B')$, i.e., $\lim_{z\to \d B'} v(z) =0$ and $\int_{B'} (dd^cv)^n <+\infty$.

\begin{remark}
The latter property often follows from the existence of  a subsolution $\ul u$. For example in geometrically interesting cases: if either  $F(u,z)= e^{\la u}$ with $\la \in \bR$, or if  $\mu$ is the volume form. 
\end{remark}

Our main result is as follows.

\begin{thm} \label{thm:main}  Suppose that $F(t,z)$ is a bounded non-negative function which is continuous and non-decreasing in the first variable and $\mu$-measurable in the second one. Let $\mu$ be a positive Radon measure satisfying \eqref{eq:local-dominated}. Suppose that there exists a bounded subsolution $\ul u$  as in \eqref{eq:subsolution-bounded}. Then there exists a bounded solution to  the Dirichlet problem \eqref{eq:Dirichlet}. 
\end{thm}

\begin{proof} We use the Perron envelope method as in \cite[Theorem~1.2]{KN21}. Let
$\cB (\vphi, \mu)$ be the set
\[\label{eq:class-B}
	 \left\{w \in PSH(M, \om) \cap L^\infty(M): (\om + dd^c w)^n \geq F(w,z)\mu, w^*_{|_{\d M}} \leq \vphi \right\},
\]
where  $w^*(x) = \limsup_{M \ni z\to x} w(z)$ for every $x\in \d M$.
Then, $\cB(\vphi, \mu)$ is non-empty as it contains $\ul u$. Let $u_0 \in C^0(\ov M)$ be a $\om$-subharmonic  solution to 
$$(\om + dd^c u_0) \wed \om^{n-1} = 0, \quad u_0 = \vphi\quad\text{on } \d M$$
(see e.g. Corollary~\ref{cor:linear-PDE}). 
By the comparison principle for $\om$-subharmonic functions we have 
\[\label{eq:basic-bound} v \leq  u_0 \quad\text{for every }v\in \cB(\vphi, \mu).\] 
Thus, the function
$$
	u (z) = \sup \{v(z) : v \in \cB(\vphi, \mu)\}
$$
is well-defined.
We know that $u^* \in PSH(M, \om) \cap L^\infty(\ov M)$ and $u= u^*$ almost everywhere, outside a pluripolar set. 
Moreover, if $v_1, v_2 \in \cB(\vphi, \mu)$, then  so is $\max\{v_1,v_2\}$. Indeed,  by an inequality of Demailly \cite{De85} we have
$$\begin{aligned}
	(\om + dd^c \max\{v_1, v_2\})^n 
&\geq 	{\bf 1}_{\{v_1>v_2\}} (\om + dd^c v_1)^n + {\bf 1}_{\{v_1\leq v_2\}} (\om + dd^c v_2)^n \\
&\geq 	{\bf 1}_{\{v_1>v_2\}} F(v_1, z) \mu + {\bf 1}_{\{v_1\leq v_2 \}} F(v_2, z) \mu \\
&=	F (\max\{v_1,v_2\},z)	\mu.
\end{aligned}$$
By this property and Choquet's lemma we can write $u= \lim_{j\to +\infty} u_j$,  where $\{u_j\}_{j\geq 1} \subset \cB(\vphi, \mu)$ is an increasing sequence. Therefore, 
$$\begin{aligned}
	(\om + dd^c u^*)^n 
&= 	\lim_{j \to +\infty} (\om + dd^c u_j)^n, \\
&\geq \lim_{j \to +\infty} F (u_j, z) \mu  \\
&=	F (u, z) \mu = F(u^*,z) \mu,
\end{aligned}$$
where the last equality follows the fact that $\mu$ does not charge  pluripolar sets.
Thus, $u=u^* \in \cB(\vphi,\mu)$ is $\om$-psh in $M$.  It also follows from the definition and \eqref{eq:basic-bound} that $\ul u\leq u \leq u_0$. Hence, $u= \vphi$ is continuous on $\d M$. 

It remains to show that $\om_u^n = F(u,z)\mu$ in $M$. To see this, let $B \subset \subset M$ be a coordinate ball in $M$. Following the same argument as in \cite[Lemma~3.7]{KN21}, given the local solvability of the Dirichlet problem (Theorem~\ref{thm:existence-local}), 
there exists $\wt u \in \cB(\vphi, \mu)$ such that $u \leq \wt u$ and $(\om+dd^c \wt u)^n = F(\wt u, z) \mu$ in this small coordinate ball.
By definition of $u$ we must have $\wt u \leq u$. So, $\wt u = u$ in $B$. In other words $\om_u^n = F(u, z) \mu$. Since $B$ is arbitrary, this proves our claim.
\end{proof}

For $F(t, z) = e^{\la t}$ with $\la>0$ we obtain a stronger statement.

\begin{cor}\label{cor:KE} Let $\la>0$. There exists  a unique  solution to the Dirichlet problem  
$$\begin{cases}
	u \in PSH(M, \om) \cap L^\infty(\ov M), \\
	\om_u^n = e^{\la u} \mu, \\
	\lim_{z\to x} u(z) = \vphi(x) \quad\text{for } x\in \d M.
\end{cases}$$  if and only if there exists a bounded subsolution. 

Moreover, if the subsolution is H\"older continuous, so is the solution. 
\end{cor}

\begin{proof} 

The uniqueness of the solution under the hypothesis of the existence of a subsolution follows by  the same arguments as in \cite[Lemma~2.3]{Ng16}. 
Next, to prove the equivalence, we only need to verify the local domination  by Monge-Amp\`ere measures  of bounded plurisubharmonic functions  for $\mu$ (see condition \eqref{eq:local-dominated}). This is straightforward.

For the second conclusion we first observe as in \cite[Lemma~6.5]{KN21} that $u$ is H\"older continuous on the boundary $\d M$. Since $u$ is bounded,  $$\om_u^n = e^{\la u}\mu   \leq C (\om + dd^c \ul u)^n.$$  Thus, the H\"older continuity of $u$ follows from the the proof of \cite[Theorem~1.4]{KN21}. 
\end{proof}

Thanks to this we can also obtain the solution of the Monge-Amp\`ere equation as the  limit of the solutions of the Monge-Amp\`ere type equations.

\begin{cor} Suppose that there exists a function $\ul u \in PSH(M, \om) \cap L^\infty(M)$ which satisfies: $\lim_{z\to x}\ul u(z) = \vphi(x)$ for $x\in \d M$ and $(\om + dd^c \ul u)^n \geq \mu$ in $M$. Then,  the sequence of solutions
$$
	(\om + dd^c u_\la)^n = e^{\la u} \mu \quad\text{with } \quad \la >0
$$
converges to a solution $u$ of $\om_u^n = \mu$ and $u = \vphi - \sup_{M} \ul u$ on $\d M$ as $\la \to 0$.
\end{cor}

\begin{proof} Let $b = \sup_{M} \ul u$. Then, $v:=\ul u - b \leq 0$ on $M$. For every $\la >0$, the function $v$ satisfies $v = \vphi - b$ on $\d M$ and 
$(\om + dd^c v)^n \geq e^{\la v} \mu$ in $M$.
Applying Theorem~\ref{thm:main} we obtain the family of solutions $\{u_\la\}_{0< \la \leq 1}$ of 
$$
	(\om +dd^c u_\la)^n = e^{\la u} \mu, \quad u_\la = \vphi - b \quad\text{on } \d M.
$$
By the domination principle the family is increasing in $\la>0$ and $u_\la \geq v$ for every $0<\la \leq 1$. Set $u = \lim_{\la \to 0} u_\la.$ Then, $u + b$ is a solution to $\om_u^n = \mu$ in $M$ and $u = \vphi$ on $\d M$.
\end{proof}

\section{Appendix}

In the proofs in this paper as well as in \cite{KN21} we use the existence and regularity of solutions of the linear elliptic equation
on a manifold with boundary. Those statements are known, for instance as consequences of general results  for harmonic maps \cite[Theorem~6]{JY93} (see also \cite[Theorem~5.3]{SY}). However, since the case of H\"older continuous boundary data seems not to be available in literature we include the proof here for the sake of completness.

Let $(\ov{M},\om)$ be a compact $n$-dimensional Hermitian  manifold with nonempty boundary $\d M$. Then $\ov{M} = M \cup \d M$, where $M$ is a (open) Hermitian manifold. Suppose that in local coordinate we have
$$
	\om = \ii g_{i\bar j} (z) dz^i \wed d\bar z^j.
$$
Define $\De_g = g^{\bar j i} \d_i \d_{\bar j}$ be the Laplace operator associated to $\om$ and denote by ${\rm dist}(z,w)$  the distance function induced by $\om$.

\begin{prop}\label{cor:linear-PDE} Let $\vphi \in C^0(\d M)$. Then, there exists a unique continuous solution to 
$$
	(\om + dd^c u) \wed \om^{n-1} =0\quad\text{in }M, \quad u= \vphi\quad\text{on } \d M. 
$$
Moreover, if $\vphi$ is H\"older continuous on $\d M$, then so is the solution $u$.
\end{prop}

%\begin{proof} 
Since  $$\De_g u = {\rm tr}_\om u = \frac{ndd^c u \wed \om^{n-1}}{\om^n},$$
we can separate the equation into two problems
$\De_g u_1 = 0$ in $M$ with $u_1 = \vphi$ on $\d M$ and $\De_g u_2 = -n$ in $M$  with  $u_2 =0$ on $\d M$. The latter solution $u_2$ is smooth by the classical PDEs. 

We are thus reduced to proving the following.

\begin{prop}\label{prop:linear-PDE}
Let $\vphi \in C^0(\d M)$. Then, there exists a unique continuous solution to 
\[\label{eq:dirichlet}	
	\De_g u = 0 \quad\text{in M}, \quad u = \vphi \quad\text{on } \d M.
\]
Moreover, if $\vphi$ is H\"older continuous on $\d M$, then so is $u$. 
\end{prop}

The existence of continuous solutions follows exactly  as in \cite[page 24-25]{GT98} by using the Perron envelope 
$$
	\cS_\vphi = \{ v \in SH_\om(M) \cap C^0(\ov M): v_{|_{\d M}} \leq \vphi\},
$$
where $SH_\om(M)$ is the set of all $\De_g$-subharmonic functions in $M$. The function $$u(x) = \sup_{v \in S_\vphi} v(x) \quad \text{for } x\in \ov M$$ is the solution to \eqref{eq:dirichlet}  by the Perron method using harmonic liftings:
\begin{lem}  Let $v \in S_\vphi$. Then there exists a function $\wt v$ called a lift of $v$ in $B$ such that $\wt v \geq v$ on $\ov M$ and  satisfying $\De_g \wt v = 0$ in $B$ and $\wt v= v$ on $\d B$.
\end{lem}
The boundary condition is satisfied since using the regularity of the boundary of the domain one can easily construct, as in   \cite{GT98}, first the local barriers, and then the global one.

We now assume further that $\vphi$ belongs to $C^{0,\al}(\ov M)$ with $0<\al <1$. Then, we wish to show that the solution also belongs to a H\"older space.

First, we will construct a H\"older continuous local barrier similarly as in \cite[Theorem~6.2]{BT76} on the coordinate half-ball at each boundary point.

\begin{lem}\label{lem:point-barrier} Suppose the origin $0 \in \d M \cap B(0,R) $,  $\rho$ is the defining function of $\d M \cap B(0,R)$ in the coordinate ball $B(0,R)$ and
$U_R = \{ z\in B(0,R): \rho (z)\leq 0 \}$ is the coordinate half-ball centered at $0$. Denote $\|\vphi\|_{\al}=c_1$ for the H\"older norm of $\vphi$ on $\d M$.  Let $0<\tau \leq \al < 1$. Then, there exists a constant $k = k(\vphi, U_R)$ and a neighborhood $W$  of $0$ such that the function
$$
	v(z) = k |\rho|^\tau(z) + c_1  |z|^{\al} + \vphi(0)
$$
is $\De_g$-superharmonic in $W \cap U_R$. Moreover, 
$$
	v(0) = \vphi(0), \quad v(x) \geq \vphi(x) \quad\text{for every } x\in \d M \cap B(0,R). 
$$
\end{lem}

\begin{proof} We compute in $B(0,R)$,
$$
	dd^c |\rho|^\tau = - \tau |\rho|^{\tau-1} dd^c \rho - \tau (1-\tau) |\rho|^{\tau-2} d\rho\wed d^c \rho,
$$
and for $\al' = \al/2$,
$$
	dd^c |z|^{2\al'} = \al' |z|^{2(\al'-1)} dd^c |z|^2 - \al'(\al'-1) |z|^{2(\al'-2)} d |z|^2 \wed d^c|z|^2.
$$
Hence,
$$dd^c v (z) \wed \om^{n-1}/\om^n \leq - k\tau(1-\tau) |\rho|^{\tau-2} |\na \rho|_g^2 + \frac{c_1 \al}{2} |z|^{\al-2}. $$
Furthermore, $|\rho(z)|= |\rho(z) - \rho(0)| \leq c_2 |z|$ for every $z \in \ov{B(0,R) }$. Since $\tau-2<0$, it implies that 
$$ \begin{aligned}
	\frac{1}{n} \De_g v (z) 
&\leq - c_3 k |z|^{\tau-2} |\na \rho|_g^2 + \frac{c_1 \al}{2} |z|^{\al-2}  \\
&= |z|^{\al-2} \left( c_1 \al/2 - c_3 k |z|^{\tau-\al} |\na \rho|_g^2\right) \\
&\leq  |z|^{\al-2} \left( c_1 \al - c_3 k R^{\tau-\al} |\na \rho|_g^2\right),
\end{aligned}$$
where $c_3 = \tau (1-\tau) c_2^{\tau-2}$. We used the fact $\tau \leq \al$ for the last inequality. Since $\rho$ is the defining function for $\d M \cap B(0,R)$, it follows that 
$|\na \rho|_g > \veps_0 >0$ in a neighborhood $W$ of $0$.  Then, we can choose large $k>0$, independent of the boundary point $0$,  so that $\De_g v \leq 0$. Finally, 
the remaining properties of $v$ hold because $\rho(x) =0$ for every $x\in \d M \cap B(0,R)$. 
\end{proof}

\begin{prop} There exists a function $\ov u(x) \in C^{0,\al}(\ov M)$ that is $\De_g$-superharmonic  in $M$ and $\ov u(\xi) = \vphi(\xi)$ for $\xi \in \d M$.
\end{prop}

\begin{proof} We first show that at each point $\xi \in \d M$ there exist a $\De_g$-superhamonic function $v_\xi$ in $M$ such that  $v_\xi \in C^{0,\al}(\ov M)$  and
\begin{align} 
&\label{eq:point-p1}	v_{\xi} (x) \geq \vphi(x) \quad\text{on } \d M, \\
&\label{eq:point-p2}	v_\xi (\xi)  = \vphi(\xi), \\
&\label{eq:point-p3}	\|v_\xi\|_{\al} \leq C \|\vphi\|_{2\al},
\end{align}
where $\| \cdot \|_\al$ denotes the $\al$-H\"older norm of the function and the constant $C$ depends only on $\ov M$ and the metric $\om$.

Without loss of generality we may assume $\xi$ is the origin and $\vphi(0) =0$.
By Lemma~\ref{lem:point-barrier} for the boundary point $0 \in \d M$ there exists a function $v = k |\rho|^\al + c_1 |z|^{\al}$ and a neighborhood $W$ of $0$ such that  $v$ satisfies \eqref{eq:point-p1}, \eqref{eq:point-p2} and \eqref{eq:point-p3} on $W \cap U_R$. Note that the constants $k,c_1$ are independent of the boundary points. We can extend this function to a global one as follows. Set $k_1 = \sup_{\d M} \vphi+1$. Then, for large $k_2\geq 1$ (depending on $k_1$ and $R$, but independent of the boundary point), the function
$$
	\wh v = \min\{k_1, k_2  v\}
$$
is $\De_g$-superharmonic on $M$ and satisfies the list of required properties. 

Now let us define
$$
	\ov u = \inf \{ v_\xi : \xi \in \d M\}. 
$$
From $|v_\xi(z) - v_\xi (w)| \leq C[{\rm dist}(z,w)]^\al$ we deduce that $|\ov u(z) -\ov u(w)| \leq  C[{\rm dist}(z,w)]^\al$. Thus, $\ov u$ is H\"older continuous $\De_g$-superharmonic, and clearly $\ov u(\xi) = \vphi (\xi)$ for every $\xi \in \d M$. This completes the proof.
\end{proof}

By the similar argument we can find a global H\"older continuous $\De_g$-subharmonic barrier $\ul u\in C^{0,\al}(\ov M)$  such that 
\[\label{eq:sub-barrier}\ul u(\xi) = \vphi(\xi) \quad\text{for }\xi \in \d M.\] 
Hence, by the maximum principle,
$\ul u \leq u \leq \ov u$  on  $\ov M.$
Consequently, we get 
\[\label{eq:holder-boundary} |u(x) - u(\xi)| \leq C [{\rm dist}( x, \xi)]^\al \quad\text{for every } x\in \ov M \text { and } \xi \in \d M.
\]
Now we are going to show the global H\"older continuity of the solution.

\begin{lem} There exists a constant $C = C(\vphi, M, \om)$ such that
\[\label{eq:holder-norm}  
	|u(x) - u(y)| \leq C [{\rm dist}(x,y)]^\al \quad\text{for every } x, y \in \ov M.
\]
\end{lem}
\begin{proof} By maximum principle we get $\inf_{\d M} \vphi \leq u \leq \sup_{\d M}\vphi$. Denote $d_x = {\rm dist}(x, \d M)$ and $d_y = {\rm dist}(y, \d M)$. Suppose that $d_y \leq d_x$, and take $x_0, y_0 \in \d M$ such that ${\rm dist}(x, x_0) = d_x$ and ${\rm dist}(y, y_0) = d_y$.

{\bf Case 1:} ${\rm dist}(x,y) \leq d_x/2$. Since $\De_g$ is uniformly elliptic we have the interior H\"older estimates (see Corollary~9.24 and Lemma~8.23 in \cite{GT98})
$$
	\|u\|_{C^{0,\al}(\ov B_{R/2})} \leq C R^{-\al} \|u\|_{L^\infty(B_R)},
$$
where $B_R \subset M$ is a coordinate ball of small radius $0< R<1$. 
Since $y \in \ov B_{d_x/2} (x) \subset B_{d_x}(x) \subset M$. Applying the interior  inequality to $u-u(x_0)$ in $B_{d_x}(x)$ we get
$$
	d_x^\al \frac{|u(x) - u(y)|}{[{\rm dist}(x,y)]^\al} \leq C \| u - u(x_0)\|_{L^\infty(B_{d_x}(x))}.
$$
By \eqref{eq:holder-boundary} the right hand side is less than $C d_x^{\al}$. It follows that
$$	|u(x) - u(y)| \leq C [{\rm dist}(x,y)]^\al.
$$

{\bf Case 2:} $d_y \leq d_x \leq 2 \,{\rm dist}(x,y)$. Then,
$$\begin{aligned}
	|u(x) - u(y)| 
&\leq 	|u(x) - u(x_0)| + |u(x_0)-u(y_0)| + |u(y)- u(y_0)| \\
&\leq	 	C(d_x^\al + [{\rm dist} (x_0,y_0)]^\al + d_y^\al) \\
&\leq		C [{\rm dist}(x,y)]^\al 
\end{aligned}$$
since ${\rm dist} (x_0 ,y_0) \leq d_x + {\rm dist}(x,y) + d_y \leq 5 \, {\rm dist}(x,y)$.
\end{proof}

%\end{proof}

%\end{proof}


\begin{thebibliography}{000000000}

\bibitem[BT76]{BT76} E. Bedford and B. A. Taylor, 
{\it The Dirichlet problem for a complex Monge-Amp\`ere operator,} Invent. math. {\bf37} (1976)  1--44.

\bibitem[BT77]{BT77}
E. Bedford\ and\ B. A. Taylor,  {\it The Dirichlet problem for an equation of complex Monge-Amp\`ere type}, inPartial differential equations and geometry (Proc. Conf., Park City, Utah, 1977), 39--50, Lecture Notes in Pure and Appl. Math., 48, Dekker, New York. 

\bibitem[BT82]{BT82} E. Bedford and B. A. Taylor, 
{\it A new capacity for plurisubharmonic functions,} Acta Math. {\bf149} (1982)  1--40.

%\bibitem[BT87]{BT87} E. Bedford\ and\ B. A. Taylor, {\it Fine topology, \v Silov boundary, and $(dd\sp c)\sp n$.} J. Funct. Anal. {\bf 72} (1987), no.~2, 225--251.

\bibitem[Be19]{Be19} R. J. Berman, {\it From Monge-Amp\`ere equations to envelopes and geodesic rays in the zero temperature limit}, Math. Z. {\bf 291} (2019), no.~1-2, 365--394. 



%\bibitem[BJZ]{BJZ} S. Benelkourchi, B. Jennane\ and\ A. Zeriahi, {\it Polya's inequalities, global uniform integrability and the size of plurisubharmonic lemniscates,} Ark. Mat. {\bf 43} (2005), 85--112. 


%\bibitem[BD12]{BD12}{R. Berman and J.-P. Demailly, \it Regularity of plurisubharmonic upper envelopes in big cohomology classes. \rm Proceedings of the Symposium ``Perspectives in Analysis, Geometry and Topology'' in honor of Oleg Viro (Stockholm University, May 2008), Eds. I.~Itenberg, B.~J\"oricke, M.~Passare, Progress in Math.\ {\bf 296}, Birkh\"auser/Springer, Boston (2012) 39--66.}


%\bibitem[BEGZ]{BEGZ} S. Boucksom, P. Eyssidieux, V. Guedj, A. Zeriahi, {\it Monge-Amp\`ere equations in big cohomology classes.} Acta. Math. {\bf 205}, (2010), 199--262.

%\bibitem[Bl93]{Bl93} Z. B\l ocki, {\it Estimates for the complex Monge-Amp\`ere operator.} Bull. Polish Acad. Sci. Math. {\bf 41} (1993)   151--157.


%\bibitem[Bl09]{Bl09} Z. B\l ocki, {\it On geodesics in the space of K\"{a}hler metrics}, in  Advances in geometric analysis, 3--19, Adv. Lect. Math. (ALM), 21, Int. Press, Somerville.

%\bibitem[BK07]{BK07} Z. B\l ocki,  S. Ko\l odziej, {\it On regularization of plurisubharmonic functions on manifolds,} Proc. Amer. Math. Soc. {\bf 135} (2007) 2089–2093.

%\bibitem[Bo12]{Bo12} S. Boucksom, {\it Monge-Amp\`ere equations on complex manifolds with boundary}, in  Complex Monge-Amp\`ere equations and geodesics in the space of K\"{a}hler metrics, 257--282, Lecture Notes in Math., 2038, Springer, Heidelberg.

\bibitem[CKNS85]{CKNS85}L. Caffarelli, J. Kohn,  L. Nirenberg, J. Spruck, {\it The Dirichlet problem for nonlinear second-order elliptic equations. II. Complex Monge-Amp\`ere, and uniformly elliptic, equations,} Comm. Pure Appl. Math. {\bf 38} (1985), 209--252.

\bibitem[Ce84]{Ce84} U. Cegrell,  {\it On the Dirichlet problem for a complex Monge-Amp\`ere operator,} Math. Z. {\bf 185} (1984), 247 - 251.

\bibitem[Ce98]{Ce98} U. Cegrell,  {\it Pluricomplex energy,}
Acta Math. {\bf 180:2} (1998) 187-217.

%\bibitem[Ce04]{Ce04} U. Cegrell, {\it The general definition of the complex Monge-Amp\`ere operator,} Ann. Inst. Fourier (Grenoble), {\bf 54} (2004)  159-179.

\bibitem[CK06]{CK06}
U. Cegrell\ and\ S. Ko\l odziej, {\it The equation of complex Monge-Amp\`ere type and stability of solutions}, Math. Ann. {\bf 334} (2006), no.~4, 713--729. 

%\bibitem[CKZ11]{CKZ11} U. Cegrell, S. Ko\l odziej\ and\ A. Zeriahi, {\it Maximal subextensions of plurisubharmonic functions}, Ann. Fac. Sci. Toulouse Math. (6) {\bf 20} (2011), Fascicule Sp\'{e}cial, 101--122. 

%\bibitem[Ch00]{Ch00} X. Chen, {\it The space of K\"{a}hler metrics}, J. Differential Geom. {\bf 56} (2000) 189--234.

%\bibitem[CDS15]{CDS2} X. X.  Chen, S.  Donaldson and S.  Sun
%{\it K\"ahler-Einstein metrics on Fano manifolds. I: Approximation of metrics with cone singularities.} J. Amer. Math. Soc. {\bf 28}  (2015), 183--197. 

%\bibitem[Ch87]{Ch87} P. Cherrier, {\it \'Equations de Monge-Amp\`ere sur les vari\'et\'es Hermitiennes compactes.} Bull. Sci. Math. {\bf 111} (2) (1987)  343--385.

%\bibitem[Chi13]{Chiose13} I. Chiose, {\it The K\"ahler rank of compact complex manifolds.} J. Geom. Anal. {\bf 26} (2016), no.~1, 603--615.

 \bibitem[CY80]{CY80} S-Y. Cheng and S-T. Yau,  {\it On the existence of a complete K\"ahler metric on
noncompact complex manifolds and the regularity of Fefferman's equation}. Comm. Pure Appl. Math., {\bf 33} (1980),  507--544.

\bibitem[CZ19]{CZ19} J. Chu\ and\ B. Zhou, {\it Optimal regularity of plurisubharmonic envelopes on compact Hermitian manifolds}, Sci. China Math. {\bf 62} (2019), no.~2, 371--380.

\bibitem[Cz10]{Cz10}
R. Czy\.{z}, {\it On a Monge-Amp\`ere type equation in the Cegrell class $\cE_\chi$}, Ann. Polon. Math. {\bf 99} (2010), 89--97.

%\bibitem [De82]{De82}{J.-P. Demailly, \it Estimations $L^2$ pour l'op\'erateur
%$\overline\partial$ d'un fibr\'e vectoriel holomorphe semi-positif au-dessus
%d'une vari\'et\'e k\"ahl\'erienne compl\`ete. \rm Ann.\ Sci.\ \'Ecole Norm.\
%Sup.\ 4e S\'er. \bf 15 \rm (1982), 457-511.}

%\bibitem [De94]{De94}{J.-P. Demailly, \it Regularization of closed positive currents of type $(1,1)$ by the flow of a Chern connection. \rm  Aspects of Math. Vol. E26, Vieweg (1994), 105--126.}

%\bibitem[DP04]{DP04} J-P. Demailly and M. Paun, {\it Numerical characterization of the K\"ahler cone of a compact K\"ahler manifold.} Ann. of Math. (2) {\bf 159} (2004), no. 3, 1247-1274.

\bibitem [De85]{De85} J.-P. Demailly, {\it Measures de Monge-Amp\`ere et caract\'erisation g\'eom\'etrique des vari\'et\'es alg\'ebraiques affines.}  Mem. Soc. Math. France (N.S.) {\bf 19}  (1985),  1-124.



%\bibitem[DDGHKZ]{DDGHKZ} J.-P. Demailly, S. Dinew, V. Guedj, P. Hiep, S. Ko\l odziej and  A. Zeriahi, {\it H\"older continuous solutions to Monge-Amp\`ere equations.} J. Eur. Math. Soc. (JEMS) {\bf 16} (2014)  619--647.

%\bibitem[DF18]{DiF} E. Di Nezza\ and\ C. Favre, Regularity of push-forward of Monge-Amp\`ere measures, Ann. Inst. Fourier (Grenoble) {\bf 68} (2018), no.~7, 2965--2979.

%\bibitem[DL15]{DL15}E. Di Nezza\ and\ C. H. Lu, {\it Generalized Monge-Amp\`ere capacities}, Int. Math. Res. Not. IMRN {\bf 2015} 7287--7322.


%\bibitem[DL17]{DL17}E. Di Nezza\ and\ C. H. Lu, {\it Complex Monge-Amp\`ere equations on quasi-projective varieties}, J. Reine Angew. Math. {\bf 727} (2017)  145--167.



%\bibitem[Di09]{dinew09} S. Dinew, {\it An inequality for mixed Monge-Amp\`ere measures.} Math. Zeit. {\bf 262} (2009), 1--15.

%\bibitem[Di16]{Di16} S. Dinew, {\it Pluripotential theory on compact Hermitian manifolds,} Ann. Fac. Sci. Toulouse Math. (6) {\bf 25} (2016), no.~1, 91--139.  

%\bibitem[Di19]{Di19} S. Dinew, {\it Lectures on Pluripotential theory on compact Hermitian manifolds,} Lecture notes in Mathematics {\bf 2246} (2019), "Complex Non-K\"ahler geometry", Springer.

\bibitem[DK12]{DK12}S. Dinew and S. Ko\l odziej, 
{\it  Pluripotential estimates on compact Hermitian manifolds.}   Adv. Lect. Math. (ALM), {\bf 21} (2012)  International Press, Boston. 


%\bibitem[DZ10]{dinew-zhang10} S. Dinew and  Z. Zhang, {\it On stability and continuity of bounded solutions of degenerate complex Monge-Amp\`ere equations over compact K\"ahler manifolds.} Adv. Math. {\bf 225} (2010), no. 1, 367--388.

%\bibitem[DNS10]{DNS} T.-C. Dinh, V.-A. Nguyen, N. Sibony, {\it Exponential estimates for plurisubharmonic functions and stochastic dynamics,} J. Differential Geom. {\bf 84} (2010) 465-488. 

%\bibitem[DN14]{DN14} T.-C. Dinh, V.-A. Nguyen, {\it Characterization of Monge-Amp\`ere measures with H\"older continuous potentials.} J. Funct. Anal. {\bf 266} (2014), no. 1, 67--84.


%\bibitem[Do99]{Do99} S. K. Donaldson, {\it Symmetric spaces, K\"{a}hler geometry and Hamiltonian dynamics}, in  Northern California Symplectic Geometry Seminar, 13--33, Amer. Math. Soc. Transl. Ser. 2, 196, Adv. Math. Sci., 45, Amer. Math. Soc., Providence, RI. 

%\bibitem[EGZ09]{EGZ09} P. Eyssidieux, V. Guedj and A. Zeriahi, {\it Singular K\"ahler-Einstein metrics.} J. Amer. Math. Soc. {\bf 22} (2009),  607--639.

%\bibitem[FTWZ16]{FTWZ} S. Fang, V. Tosatti, B. Weinkove, T. Zheng, {\it Inoue surfaces and the Chern-Ricci flow,} J. Funct. Anal. {\bf 271} (2016), no.~11, 3162--3185.

%\bibitem[FY08]{FY08} J.-X. Fu and S.-T. Yau,  {\it The theory of superstring with flux on non-K\"ahler manifolds and  the complex Monge-Amp\`ere equation.} J. Differential Geom. {\bf 98} (2008), 369--428.

%\bibitem[FWW10a]{fww1} J. Fu, Z. Wang\ and\ D. Wu, {\it Form-type Calabi-Yau equations.} Math. Res. Lett. {\bf 17} (2010),  887--903. 

%\bibitem[FWW10b]{fww2} J. Fu, Z. Wang\ and\ D. Wu, {\it Form-type Calabi-Yau equations on K\"ahler manifolds of nonnegative orthogonal bisectional curvature.} Preprint, arXiv: 1010.2022.

%\bibitem[GS80]{GS80} T. W. Gamelin\ and\ N. Sibony, {\it Subharmonicity for uniform algebras}, J. Functional Analysis {\bf 35} (1980), no.~1, 64--108. 


%\bibitem[Gau77]{gauduchon77} P. Gauduchon, {\it La th\'eor\`eme de l'excentricit\'e nulle.} C. R. Acad. Sci. Paris {\bf 285} (1977), 387--390.

\bibitem[GLZ19]{GLZ19}V. Guedj, C. H. Lu\ and\ A. Zeriahi, {\it Plurisubharmonic envelopes and supersolutions}, J. Differential Geom. {\bf 113} (2019), no.~2, 273--313.


%\bibitem[Gil11]{gill11} M. Gill, {\it Convergence of the parabolic complex Monge-Amp\`ere equation on compact Hermitian manifolds.} Comm. Anal. Geom. {\bf 19} (2011),  277--303.

%\bibitem[Gil13]{gill13} M. Gill, {\it The Chern-Ricci flow on smooth minimal models of general type.} Preprint arXiv: 1307.0066v1.

%\bibitem[Gu98]{Gu98} B. Guan, {\it The Dirichlet problem for complex Monge-Amp\`ere equations and regularity of the pluri-complex Green function}, Comm. Anal. Geom. {\bf 6} (1998)   687--703.

%\bibitem[GL10]{GL10} B. Guan and Q. Li, {\it Complex Monge-Amp\`ere equations and totally real submanifolds.} Adv. Math. {\bf 225} (2010)   1185--1223.

%\bibitem[GL21a]{GL21a} V. Guedj, C.-H. Lu, {\em Quasi-plurisubharmonic envelopes 2: Bounds on Monge-Amp\`ere volumes,} Preprint, {\bf arXiv:} 2106.04272.

%\bibitem[GL21b]{GL21b} V. Guedj, C.-H. Lu,  {\it Quasi-plurisubharmonic envelopes 3: Solving Monge-Amp\`ere equations on hermitian manifolds}, Preprint, {\bf arXiv:} 2107.01938. 


\bibitem[GT]{GT98}
D. Gilbarg\ and\ N. S. Trudinger, {\it Elliptic partial differential equations of second order}, reprint of the 1998 edition, Classics in Mathematics, Springer-Verlag, Berlin, 2001
%\bibitem[GZ05]{GZ1} V.Guedj and A.Zeriahi, {\it Intrinsic capacities on compact K\"ahler manifolds} J. Geom. Anal. {\bf 15} (2005),%607-639.



\bibitem[GKZ08]{GKZ08} V. Guedj, S. Ko\l odziej, A. Zeriahi, {\it H\"older continuous solutions to Monge-Amp\`ere equations,} Bull. Lond. Math. Soc. {\bf 40} (2008)  1070-1080.

\bibitem[GL21]{GL21a} V. Guedj, C.-H. Lu, {\em
Quasi-plurisubharmonic envelopes 2: Bounds on Monge-Amp\`ere volumes,}
Preprint, {\bf arXiv:} 2106.04272.
%\bibitem[GZ07]{GZ07} V. Guedj and A. Zeriahi, {\it The weighted Monge-Amp\`ere energy of quasiplurisubharmonic functions,} J. Funct. Anal. {\bf 250} (2) (2007), 442-482.

%\bibitem[GZ17]{GZ-book} V. Guedj\ and\ A. Zeriahi, {\it Degenerate complex Monge-Amp\`ere equations}, EMS Tracts in Mathematics, 26, European Mathematical Society (EMS), Z\"{u}rich, 2017.

\bibitem[HQ19]{HQ19} L. M. Hai\ and\ V. V. Quan, H\"{o}lder continuity for solutions of the complex Monge-Amp\`ere type equation, J. Math. Anal. Appl. {\bf 494} (2021), no.~1, Paper No. 124586, 14 pp. 

\bibitem[HQ21]{HQ21}
L. M. Hai\ and\ V. V. Quan, {\it Existence and H\"{o}lder continuity to solutions of the complex Monge--Amp\`ere-type equations with measures vanishing on pluripolar subsets,} Internat. J. Math. {\bf 32} (2021), no.~14, Paper No. 2150099, 19 pp.

%\bibitem[Hi10]{hiep} H.-H. Pham,  {\it H\"older continuity of solutions to the Monge-Amp\`ere equations on compact K\"ahler manifolds.} Ann. Inst. Fourier (Grenoble) {\bf 60} (2010), no. 5, 1857--1869.


\bibitem[JY93]{JY93}
J. Jost\ and\ S.-T. Yau, {\em A nonlinear elliptic system for maps from Hermitian to Riemannian manifolds and rigidity theorems in Hermitian geometry,} Acta Math. {\bf 170} (1993)   221--254.

%\bibitem[Ki78]{kis} C. O. Kiselman, {\it The partial Legendre transformation for plurisubharmonic functions.} Invent. math. {\bf 49} (1978),  137--148.

\bibitem[Ko95]{Ko95}
S. Ko\l odziej, {\em The range of the complex Monge-Amp\`ere operator. II,} Indiana Univ. Math. J. {\bf 44} (1995)  765--782. 

%\bibitem[Ko96]{Ko96} S. Ko\l odziej, {\it Some sufficient conditions for solvability of the Dirichlet problem for the complex Monge-Amp\`ere operator}, Ann. Polon. Math. {\bf 65} (1996)  11-21.


%\bibitem[Ko98]{Ko98} S. Ko\l odziej, {\it The complex Monge-Amp\`ere equation,} Acta Math. {\bf 180} (1998)  69-117.

\bibitem[Ko00]{Ko00}
S. Ko\l odziej, {\it Weak solutions of equations of complex Monge-Amp\`ere type}, Ann. Polon. Math. {\bf 73} (2000), no.~1, 59--67. 

%\bibitem[Ko03]{Ko03} S. Ko\l odziej,  {\it The Monge-Amp\`ere equation on compact K\"ahler manifolds.} Indiana Univ. Math. J. {\bf 52} (2003)   667--686.

%\bibitem[Ko05]{K05} S. Ko\l odziej, {\it The complex Monge-Amp\`ere equation and pluripotential theory.} Memoirs Amer. Math. Soc. {\bf 178} (2005)  pp. 64.

%\bibitem[Ko08]{Ko08} S. Ko\l odziej,  {\it H\"older continuity of solutions to the complex Monge-Amp\`ere equation with the right hand side in $L^p$. The case of compact K\"ahler manifolds.} Math. Ann. {\bf 342} (2008)  379--386.

\bibitem[KN15]{KN1} S. Ko\l odziej and N.-C. Nguyen, {\it Weak solutions to the complex Monge-Amp\`ere equation on Hermitian manifolds.} Analysis, complex geometry, and mathematical physics: in honor of Duong H. Phong, Contemporary Mathematics, vol. 644 (American Mathematical Society,
Providence, RI, 2015)  141-158.



%\bibitem[KN16b]{KN3}S. Ko\l odziej\ and\ N. C. Nguyen, {\it Weak solutions of complex Hessian equations on compact Hermitian manifolds}, Compos. Math. {\bf 152} (2016), no.~11, 2221--2248.

%\bibitem[KN18]{KN4}S. Ko\l odziej\ and\ N.-C. Nguyen,  {\it H\"older continuous solutions of the Monge-Amp\`ere equation on compact Hermitian manifolds,}  Ann. Inst. Fourier (Grenoble) {\bf 68} (2018)  2951--2964. 


%\bibitem[KN19a]{KN2} S. Ko\l odziej\ and\ N.-C. Nguyen, {\it Stability and regularity of solutions of the Monge-Amp\`ere equation on Hermitian manifolds,} Adv. Math. {\bf 346} (2019)  264--304. 


%\bibitem[KN19b]{KN5} S. Ko\l odziej\ and\ N.-C. Nguyen, {\it A remark on the continuous subsolution problem for the complex Monge-Amp\`ere equation,}  Acta Math Vietnam (2019). https://doi.org/10.1007/s40306-019-00347-0

%\bibitem[KN6]{KN6} S. Ko\l odziej, N.-C. Nguyen, {\it An inequality between complex Hessian measures of H\"older continuous $m-$subharmonic functions and capacity.} Geometric Analysis, In Honor of Gang Tian's 60th Birthday, Progress in Mathematics series by Birkhauser. Editors: Jingyi Chen, Peng Lu, Zhiqin Lu, and Zhou Zhang,  to appear.

%\bibitem[KN21a]{KN7} S. Ko\l odziej\ and\ N.-C. Nguyen, {\em Continuous solutions to Monge-Amp\`ere equations on Hermitian manifolds for measures dominated by capacity,} Calc. Var. Partial Differential Equations {\bf 60} (2021)  Paper No. 93, 18 pp. 

\bibitem[KN21]{KN21}
S. Ko\l odziej\ and\ N.-C. Nguyen, {\em The Dirichlet problem for the Monge-Amp\`ere equation on Hermitian manifolds with boundary,} preprint. {\bf arXiv:} 2112.10042.


%\bibitem[KT19]{KTo} S. Ko\l odziej and V. Tosatti, {\it Morse-type integrals on non-K\"ahler manifolds,} arXiv:1906.09614,  to appear in Pure Appl. Math. Q. 2020.


%\bibitem[LPT20]{LPT20} Chinh, H. Lu, T.-T. Phung, T.-D. To, {\it Stability and H\"older regularity of solutions to complex Monge-Amp\`ere equations on compact Hermitian manifolds}, {\bf arXiv:} 2003.08417. to appear in Annales de l'Institut Fourier.

%\bibitem[Ni17]{Ni17} X. Nie, {\it Weak solution of the Chern-Ricci flow on compact complex surfaces,} Math. Res. Lett. {\bf 24} (2017)  1819--1844. 


%\bibitem[Ng14]{Ng14} N.-C. Nguyen, {\it Weak solutions to the complex Hessian equation,} Ph.D thesis, Jagiellonian University 2014. 

\bibitem[Ng16]{Ng16} N.-C. Nguyen, {\it The complex Monge-Amp\`ere type equation on compact Hermitian manifolds and Applications.} 
Adv. Math. {\bf 286} (2016)  240-285.

%\bibitem[Ng17]{Ng17} N.-C. Nguyen, {\it On the H\"older continuous subsolution problem for the complex Monge-Amp\`ere equation,} Calc. Var. Partial Differential Equations {\bf 57} (2018)  Art. 8, 15 pp.

\bibitem[Ng18]{Ng18} N.-C. Nguyen, {\it On the H\"{o}lder continuous subsolution problem for the complex Monge-Amp\`ere equation, II,} Anal. PDE {\bf 13} (2020)  435--453.

%\bibitem[Nie13]{nie13} X. Nie, {\it Regularity of a complex Monge-Amp\`ere equation on Hermitian manifolds.} Comm. Anal. Geom. {\bf 22} (2014), no. 5, 833--856.

%\bibitem[PSS12]{pss12} D. H. Phong, J. Song and  J. Sturm, {\it Complex Monge-Amp\`ere equations.} Surveys in differential geometry. Vol. XVII, 327-410, Surv. Differ. Geom., {\bf 17}, Int. Press, Boston, MA, 2012.


%\bibitem[Po13]{popovici13} D. Popovici, {\it Aeppli Cohomology Classes Associated with Gauduchon Metrics on Compact Complex Manifolds.} Bull. Soc. Math. France {\bf 143} (2015),  763--800.


%\bibitem[Se92]{Se92} S. Semmes, {\it Complex Monge-Amp\`ere and symplectic manifolds}, Amer. J. Math. {\bf 114} (1992) 495--550. 

\bibitem[SY97]{SY} R. Schoen\ and\ S. T. Yau, {\it Lectures on harmonic maps}, Conference Proceedings and Lecture Notes in Geometry and Topology, II, International Press, Cambridge, MA, 1997. 

%\bibitem[ST07]{ST1} J. Song, and G. Tian, {\it The K\"ahler-Ricci flow on surfaces of positive Kodaira dimension.}  Invent. math. {\bf 170}, (2007), 609--653.

%\bibitem[ST12]{ST3} J. Song\ and\ G. Tian, {\it Canonical measures and K\"ahler-Ricci flow,} J. Amer. Math. Soc. {\bf 25} (2012),  303--353.


%\bibitem[ST17]{ST2} J. Song and G. Tian, {\it The K\"ahler-Ricci flow through singularities.}  Invent. math. {\bf 207}, (2017), 519--595. 


%\bibitem[SzTo11]{SzTo11} G. Sz\'ekelyhidi\ and\ V. Tosatti, {\it Regularity of weak solutions of a complex Monge-Amp\`ere equation.} Anal. PDE {\bf 4} (2011),  369--378.

%\bibitem[SzTW]{SzTW} G. Sz\'ekelyhidi, \ V. Tosatti, \ B. Weinkove, {\it Gauduchon metrics with prescribed volume form,} arXiv:1503.04491. [math.DG].


%\bibitem[TZh06]{tian-zhang06} G. Tian and  Z. Zhang, {\it On the K\"ahler-Ricci flow on projective manifolds of general type.} Chinese Ann. Math. Ser. B {\bf 27} (2006),  179--192.

%\bibitem[To09]{To1} V. Tosatti, {\it Limits of Calabi-Yau metrics when the K\"ahler class degenerates.} J. Eur. Math. Soc. (JEMS) {\bf 11} (2009),  755--776. 

%\bibitem[To10]{To2} V. Tosatti, {\it Adiabatic limits of Ricci-flat K\"ahler metrics, J. Differential Geom.} {\bf 84} (2010),  427--453.


\bibitem[To19]{To19} V. Tosatti, {\it Regularity of envelopes in K\"{a}hler classes}, Math. Res. Lett. {\bf 25} (2018), no.~1, 281--289.

%\bibitem[TW10a]{TW10a} V. Tosatti and B. Weinkove,  {\it Estimates for the complex Monge-Amp\`ere equation on Hermitian and balanced manifolds.} Asian J. Math. {\bf 14} (2010),  19--40.

%\bibitem[TW10b]{TW10b} V. Tosatti and B. Weinkove, {\it The complex Monge-Amp\`ere equation on compact Hermitian manifolds.} J. Amer. Math. Soc. {\bf 23} (2010)  1187--1195.

%\bibitem[TW12]{TW12}
%V. Tosatti\ and\ B. Weinkove, {\it Plurisubharmonic functions and nef classes on complex manifolds,} Proc. Amer. Math. Soc. {\bf 140} (2012), no.~11, 4003--4010.

%\bibitem[TW13]{TW12b} V. Tosatti and B. Weinkove,  {\it The Chern-Ricci flow on complex surfaces.} Compos. Math. {\bf 149} (2013), no. 12, 2101--2138.

%\bibitem[TW15]{TW12a} V. Tosatti and B. Weinkove, {\it On the evolution of a Hermitian metric by its Chern-Ricci form.} J. Differential Geom. {\bf 99} (2015),  125--163.

%\bibitem[TW17]{TW13} V. Tosatti and B. Weinkove,  {\it The Monge-Amp\`ere equation for $(n-1)$-plurisubharmonic functions on a compact K\"ahler manifold.} J. Amer. Math. Soc. {\bf 30} (2017),  311--346.

%\bibitem[TWY15]{TWYang15} V. Tosatti and B. Weinkove and W. Yang, {\it Collapsing of the Chern-Ricci flow on elliptic surfaces.}  Math. Ann. {\bf 362} (2015),  1223--1271.


%\bibitem[To18]{To18} D. T\^{o}, {\it Regularizing properties of complex Monge-Amp\`ere flows II: Hermitian manifolds,} Math. Ann. {\bf 372} (2018), no.~1-2, 699--741.

%\bibitem[Vu16]{viet16} D.-V, Vu {\it Complex Monge-Amp\`ere equation for measures supported on real submanifolds,} preprint, arXiv:1608.02794. to appear in Math. Ann

%\bibitem[Vu19]{viet19} D.-V. Vu, {\it Equilibrium measures of meromorphic self-maps on non-K\"{a}hler manifolds,} Trans. Amer. Math. Soc. {\bf 373} (2020), no.~3, 2229--2250.


%\bibitem[Yau76]{Y} S.-T. Yau, {On the Ricci curvature of a compact K\"ahler manifold and the complex Monge-Amp\`ere equation.} Comm. Pure Appl. Math. {\bf 31}   \rm (1978), 339--411.

%\bibitem[Zh17]{Zh17} T. Zheng, {\it The Chern-Ricci flow on Oeljeklaus-Toma manifolds,} Canad. J. Math. {\bf 69} (2017), no.~1, 220--240.

%\bibitem[Z21]{Z21} A. Zeriahi, {\it Remarks on the modulus of continuity of subharmonic functions}, \rm {\bf arXiv:} 2007.08399.

\end{thebibliography}
\end{document}